\documentclass[11pt, reqno]{amsart}
\usepackage{hyperref}
\usepackage{amsmath}
\usepackage{amscd}
\usepackage{amsthm}
\usepackage{amssymb}
\usepackage{latexsym}
\usepackage{eufrak}
\usepackage{euscript}
\usepackage{epsfig}
\usepackage{graphics}
\usepackage{array}
\usepackage{enumerate}
\usepackage{dsfont}
\usepackage{color}
\usepackage{wasysym}
\usepackage{graphicx}
\usepackage{amsfonts}
\usepackage{mathrsfs}
\usepackage{dsfont}
\usepackage{indentfirst}
\usepackage{esint}

\newtheorem{theoremlet}{Theorem}

\newtheorem{thm}{Theorem}[section]

\newtheorem{corolet}[theoremlet]{Corollary}
\newtheorem{coro}[thm]{Corollary}
\newtheorem{lemma}[thm]{Lemma}
\newtheorem{pro}[thm]{Proposition}

\newtheorem*{thmA}{Theorem A}
\newtheorem*{thmB}{Theorem B}
\newtheorem*{corC}{Corollary C}
\newtheorem*{corD}{Corollary D}
\newtheorem*{corE}{Corollary E}

\theoremstyle{definition}
\newtheorem{definition}[thm]{Definition}
\newtheorem{example}[thm]{Example}
\newtheorem{remark}[thm]{Remark}
\newtheorem{quest}[thm]{Question}

\newcommand{\Hmm}[1]{\leavevmode{\marginpar{\tiny%
$\hbox to 0mm{\hspace*{-0.5mm}$\leftarrow$\hss}%
\vcenter{\vrule depth 0.1mm height 0.1mm width \the\marginparwidth}%
\hbox to 0mm{\hss$\rightarrow$\hspace*{-0.5mm}}$\\\relax\raggedright
#1}}}

\begin{document}
\title[Eigenvalue ratios of nonnegatively curved graphs]{Eigenvalue ratios of nonnegatively \\curved graphs}
\author{Shiping Liu}
\address{School of Mathematical Sciences, University of Science and Technology of China, 96 Jinzhai Road, Hefei 230026, Anhui Province, China}
\email{spliu@ustc.edu.cn}
\author{Norbert Peyerimhoff}
\address{Department of Mathematical Sciences, Durham University, DH1 3LE Durham, United Kingdom}
\email{norbert.peyerimhoff@durham.ac.uk}
\begin{abstract}
  We derive an optimal eigenvalue ratio estimate for finite weighted
  graphs satisfying the curvature-dimension inequality
  $CD(0,\infty)$. This estimate is independent of the size of the
  graph and provides a general method to obtain higher order spectral
  estimates. The operation of taking Cartesian products is shown to be
  an efficient way for constructing new weighted graphs satisfying
  $CD(0,\infty)$. We also discuss a higher order Cheeger constant
  ratio estimate and related topics about expanders.
\end{abstract}


\maketitle

\section{Introduction}

\subsection{Some historical background}

Exploring the influence of eigenvalues on graph structures is one of
the central topics in spectral graph theory, see e.g. \cite{AM1985},
\cite{Chung89}, \cite{Chung}, \cite{CGY}, \cite{Mohar91}. In this
area, the first nonzero Laplacian eigenvalue and the Cheeger constant
play a fundamentally important role and their close relations have
found tremendous applications in both theoretical and applied fields
like the study of expander graphs.

Let $G=(V,E)$ be a non-oriented and connected graph with vertex set $V$ and edge set $E$. For
simplicity, we consider in this subsection the special case of the
\emph{normalized} Laplacian $\Delta = D^{-1} A- {\rm{Id}}$ (where $D$
is a diagonal matrix containing the vertex degrees and $A$ is the
adjacency matrix). \emph{Cheeger's isoperimetric constant} is defined
by
\begin{equation} \label{eq:classcheeg}
h(G) = \inf_{\emptyset \neq S \subset V} \frac{|E(S,V \backslash S)|}{\min\{{\mu(S), \mu(V \backslash S)\}}},
\end{equation}
where $E(S_1,S_2)$ is the set of all edges connecting a vertex in
$S_1$ with a vertex in $S_2$ and $\mu(S) = \sum_{x \in S} d_x$, $d_x$
equals vertex degree of $x \in V$. The classical \emph{Cheeger
  inequality} states the following relation between $h(G)$ and the
first nonzero eigenvalue $\lambda_2(G) > 0$ of $-\Delta$:
$$\frac{h^2(G)}{4} \le \lambda_2(G) \le 2h(G). $$
Recently, there were two major developments in this area:
\begin{itemize}
\item higher order Cheeger constants $h_k(G)$ and higher order Cheeger
  inequalities
  \begin{equation} \label{eq:highordcheeg}
  \frac{C h_k^2(G)}{k^4} \le \lambda_k(G) \le
  2h_k(G),
  \end{equation}
  with a universal constant $C > 0$ by Miclo \cite{Miclo2008} and
  Lee, Oveis Gharan and Trevisan \cite{LGT2013}, where $h_2(G)$ agrees with
  the classical Cheeger constant $h(G)$,
\item an improved Cheeger inequality
  \begin{equation} \label{eq:impCheeg}
  h_2(G) \le C k \frac{\lambda_2(G)}{\sqrt{\lambda_k(G)}}
  \end{equation}
  with a universal constant $C > 0$ by Kwok, Lau, Lee, Oveis Gharan
  and Trevisan \cite{KLLGT2013}.
\end{itemize}
\begin{remark}
When the gap between $\lambda_2$ and $\lambda_k$ is large, (\ref{eq:impCheeg}) gives a lower bound of $\lambda_2(G)$ linear in $h_2(G)$. Another such kind of result is due to Miclo \cite{Miclo1999}, whic asserts that
\begin{equation}\label{eq:Miclo}
\frac{h_2(G)}{ {\rm diam}(G)}\leq \lambda_2(G),
\end{equation}
where $ {\rm diam}(G)$ denotes the diameter of the graph $G$.
\end{remark}

In the manifold context, another classical spectral result is Buser's
inequality \cite{Buser82}, providing under the additional assumption of
non-negative Ricci-curvature an estimate of $\lambda_1$ from above by
$h^2$, which depended on the dimension of the manifold. Later, a
dimension independent Buser type estimate was proved by Ledoux \cite{Ledoux04}
in the manifold setting.

To formulate such a result in the graph theoretical context, a
suitable curvature notion for graphs is required. Klartag, Kozma,
Ralli and Tetali \cite{KlarKozma} proved for finite graphs satisfying
the \emph{curvature-dimension condition $CD(0,\infty)$} such a Buser
type inequality:
\begin{equation} \label{eq:Busineq}
\lambda_2(G) \le C d_G h_2^2(G),
\end{equation}
with a universal constant $C > 0$, where $d_G$ denotes the maximal
vertex degree of $G$.

In this article, we combine \eqref{eq:impCheeg} and \eqref{eq:Busineq}
(in the more general setting of weighted graphs) to derive an
eigenvalue ratio result and discuss its optimality. This result has
various applications like higher order Buser estimates, an higher
order eigenvalue-diameter estimate, and higher order Cheeger
constant-ratio estimates. We also provide a discussion of the
underlying curvature notion.

\subsection{The curvature-dimension condition $CD(K,n)$}

This notion goes back to Bakry-{\'E}mery and was first studied by \cite{Schmuckenschlaeger} and \cite{LinYau}.  Since all results in this paper require such a
curvature-dimension condition, we now provide a motivation and brief
introduction into this notion.

In the setting of an $n$-dimensional Riemannian manifold $(M,g)$,
Bochner's formula implies the following inequality relating
Ricci curvature and the Laplacian:
$$ \frac{1}{2} \Delta \Vert \nabla f \Vert^2 - \langle
\nabla f, \nabla \Delta f \rangle \ge \frac{1}{n} \left( \Delta f
\right)^2 + {\rm{Ric}}(\nabla f). $$
Assuming ${\rm{Ric}} \ge K$, this inequality can be transformed with
the Bakry-\'Emery $\Gamma$-calculus, defined by
\begin{equation} \label{eq:Gamma1}
2\Gamma(f,g) = \Delta (fg) - f (\Delta g) - (\Delta f) g = 2 \langle
\nabla f, \nabla g \rangle
\end{equation}
and
\begin{equation} \label{eq:Gamma2}
2\Gamma_2(f,g) = \Delta \Gamma(f,g) - \Gamma(f, \Delta g) - \Gamma(\Delta f, g)
\end{equation}
into
\begin{equation} \label{eq:curvature-dimension-inequality}
\Gamma_2(f) \ge \frac{1}{n} \left( \Delta f \right)^2 + K \Gamma(f),
\end{equation}
where $\Gamma(f)= \Gamma(f,f)$ and $\Gamma_2(f) = \Gamma_2(f,f)$. Note
that \eqref{eq:curvature-dimension-inequality} involves a curvature
parameter $K$ and a dimension parameter $n$. This inequality makes
also sense in the graph theoretical setting and if it is satisfied for
all functions $f$, we say that the graph satisfies the
\emph{curvature-dimension inequality $CD(K,n)$}. In this paper, we are
particular concerned with graphs satisfying the condition
$CD(0,\infty)$. This condition holds for all abelian Cayley graphs (see \cite[Proposition 1.6 (1)]{LinYau}, \cite[Theorem 2.3]{KlarKozma}) but
not for trees of degree $\ge 3$ (see, e.g., \cite[Remark 16]{JostLiu14}). As a general guideline,
$CD(0,\infty)$ requires that every vertex is contained in sufficiently
many short cycles, which can be understood as a kind of local
connectivity.

\subsection{General setting}

Our results are given in the more general setting of weighted graphs
$(G,\mu)$, where $G=(V, E, w)$ is an undirected weighted finite
connected graph and $V$ and $E$ are the sets of vertices and edges,
respectively. Edge weights on $G$ are assigned via the symmetric
function $w: V \times V \to \mathbb{R}_{\ge 0}$ with
$w_{xy} = w_{yx} > 0$ iff $x \sim y$. We say the graph $G$ is
\emph{unweighted}, if $w_{xy}=1$ for any $x \sim y$, in short,
$w = \mathbf{1}_{E}$. Moreover, we assign a positive measure on the
vertex set $V$ via the function $\mu: V \to \mathbb{R}_{>0}$. Let
$d_x:=\sum_{y,y\sim x}w_{xy}$ be the degree of a vertex $x$ and
$d_G:=\max_{x\in V}d_x$ be the maximal degree of the graph $G$. For
any function $f:V\rightarrow \mathbb{R}$ and any vertex $x\in V$, the
associated Laplacian $\Delta$ is defined as
\begin{equation*}
\Delta f(x):=\frac{1}{\mu(x)}\sum_{y,y\sim x}w_{xy}(f(y)-f(x)).
\end{equation*}
This operator is called $\mu$-Laplacian in \cite{BHLLMY13}.

The normalized and the non-normalized Laplacian are contained in this
general setting as the following special cases:
\begin{itemize}
\item non-normalized Laplacian: if $\mu(x)=1$ $\forall\,\,x\in V$
  ($\mu=\mathbf{1}_V$ for short);
\item normalized Laplacian: if $\mu(x)=d_x$ $\forall\,\,x\in V$
  ($\mu=\mathbf{d}_V$ for short).
\end{itemize}
Note that the curvature condition $CD(0,\infty)$ of a graph $(G,w,\mu)$
depends on the choice of Laplacian via the formulas \eqref{eq:Gamma1}, \eqref{eq:Gamma2}, and \eqref{eq:curvature-dimension-inequality}.

The following two quantities $D_G^{non}$ and $D_G^{nor}$ appear
naturally in our arguments.
\begin{equation*}
  D_G^{non}:=\max_{x\in V}\frac{\sum_{y,y\sim x}w_{xy}}{\mu(x)}, \,\,\text{ and }\,\, D_G^{nor}:=\max_{x\in V}\max_{y,y\sim x}\frac{\mu(x)}{w_{xy}}.
\end{equation*}
Observe that on an unweighted graph, in either of the cases
$\mu=\mathbf{1}_V$ or $\mu=\mathbf{d}_V$ we always have
$D_G^{non}D_G^{nor}=d_G$.

We order the eigenvalues of $\Delta$ with multiplicities by
$$ 0 =\lambda_1(G,\mu) < \lambda_2(G,\mu) \le \cdots \le \lambda_{|V|}(G,\mu)
\le 2 D_G^{non},$$
where $\lambda \ge 0$ is an eigenvalue if there exists a non-zero
solution of $\Delta f + \lambda f = 0$.

\subsection{Results}

Combining the improved Cheeger inequality and Buser's inequality leads
to the following eigenvalue ratio:

\begin{theoremlet}\label{thm:eigenvalueratioIntroduction}
 For any finite graph $(G,\mu)$ satisfying
  $CD(0,\infty)$ and any natural number $k \ge 2$, we have
  \begin{equation}\label{eq:eigratio}
  \lambda_k(G,\mu) \le \left(\frac{20\sqrt{2}e}{e-1}\right)^2
    D_G^{non}D_G^{nor}k^2\lambda_2(G,\mu).
  \end{equation}
\end{theoremlet}

It is natural to ask about the optimality of this result: \emph{are the
curvature condition and the dependence on the $D_G^{non}D_G^{nor}$
necessary and can the quadratic term $k^2$ in \eqref{eq:eigratio} be
improved?}

\begin{itemize}
\item The unweighted dumbbell graph in Example \ref{example:dumbbell} provides a
  counterexample to \eqref{eq:eigratio} if we drop the curvature
  condition $CD(0,\infty)$.
\item Weighted triangles and tetrahedra in Examples \ref{example:triangle2} and \ref{example:tetrahedron2}
  show that the factor $D_G^{non}D_G^{nor}$ cannot be dropped.
\item Unweighted cycles in Example \ref{example:Cycle} show that the quadratic
    exponent in \eqref{eq:eigratio} is optimal.
\end{itemize}

Another natural question is: \emph{How restrictive is the $CD(0,\infty)$
condition?} It is possible to produce many new examples from given
graphs satisfying $CD(0,\infty)$ by taking Cartesian products due to
the following fundamental result.

\begin{theoremlet}\label{thm:CartesianProductIntroduction}
  If $(G_1,\mathbf{1}_{V_1})$ and
  $(G_2,\mathbf{1}_{V_2})$ satisfy $CD(K_1, n_1)$ and
  $CD(K_2, n_2)$ respectively, then
  $(G_1\times G_2, \mathbf{1}_{V_1\times V_2})$ satisfies
  $CD(K_1\wedge K_2, n_1+n_2)$.
\end{theoremlet}

Here we used the notion $K_1\wedge K_2:=\min\{K_1,K_2\}$. The above
estimate is optimal at least for the Cartesian product of a graph $G$
with itself (Remark \ref{rmk:optimalCartesian}). Theorem
\ref{thm:CartesianProductIntroduction} can be extended to include the
case of regular graphs with normalized Laplacian operators (Remark
\ref{rmk:extendCartesian}). In particular, the property of satisfying
$CD(0,\infty)$ is preserved when taking Cartesian product in many
cases.

\smallskip

Theorem \ref{thm:eigenvalueratioIntroduction} can be used as a general
source to derive various interesting higher order estimates between
geometric invariants and spectra. Of particular interest are the
higher order Cheeger constants $h_k(G,\mu)$ defined as follows: For a
given $(G, \mu)$, the expansion $\phi_{w,\mu}(S)$ of a nonempty subset
$S$ of $V$ is given by
\begin{equation*}
  \phi_{w,\mu}(S):=\frac{|E(S, V\setminus S)|_w}{\mu(S)},
\end{equation*}
where $|E(S, V\setminus S)|_w:=\sum_{x\sim y, x\in S, y\not\in
  S}w_{xy}$ and $\mu(S):=\sum_{x\in S}\mu(x)$.

\begin{definition}[Higher order Cheeger constants \cite{Miclo2008,LGT2013}]\label{defn:multi-wayIso}
  For a natural number $k$, the $k$-th Cheeger constant of
  $(G,\mu)$ is defined as
  \begin{equation*}
    h_k(G,\mu):=\min_{S_1,\ldots,S_k}\max_{1\leq i\leq k}\phi_{w,\mu}(S_i),
  \end{equation*}
  where the minimum is taken over all collections of $k$ non-empty,
  mutually disjoint subsets $S_1,\ldots, S_k$, i.e., all
  $k$-subpartitions of $V$.
\end{definition}

Note that $h_2(G,\mu)$ coincides with the classical Cheeger constant
and that $h_{k}(G,\mu)\leq h_{k+1}(G,\mu)$.

\smallskip

We use Theorem \ref{thm:eigenvalueratioIntroduction} to derive the
following higher order Buser inequality:

\begin{corolet}\label{cor:higherBuser}
  For any graph $(G, \mu)$ satisfying $CD(0,\infty)$ and any
  natural number $k$, we have
  $$
    h_k(G,\mu)\geq h_2(G,\mu)\geq \frac{(e-1)^2}{40\sqrt{2}e^2}\frac{1}{D_G^{nor}\sqrt{D_G^{non}}}\frac{\sqrt{\lambda_k(G,\mu)}}{k}.
  $$
\end{corolet}

Combining the inequalities of Alon and Milman \cite{AM1985} and Theorem
\ref{thm:eigenvalueratioIntroduction} leads to the following
higher order eigenvalue-diameter estimate:

\begin{corolet} Let $(G,\mathbf{1}_V)$ be an unweighted finite graph
  satisfying $CD(0,\infty)$. Then we have for any $k \ge 2$:
  $$
    {\rm diam}(G)\leq \frac{80 e}{e-1} d_G \log_2|V| \frac{k}{\sqrt{\lambda_k(G,\mu)}}.
  $$
\end{corolet}

This result compares nicely with the celebrated
Cheng estimate (\cite[Corollary 2.2]{Cheng75})
\begin{equation} \label{eq:cheng}
{\rm diam}(M) \leq \sqrt{2n(n+4)} \frac{k}{\sqrt{\lambda_k(M,g)}}
\end{equation}
for compact Riemannian manifolds $(M,g)$ with nonnegative Ricci
curvature.

In combination with the higher order Cheeger inequalities in
\cite{LGT2013}, Theorem \ref{thm:eigenvalueratioIntroduction} implies
the following higher order Cheeger constant-ratio estimate.

\begin{corolet}\label{cor:isoperimetricratioIntroduction}
  There exists a universal constant $C$ such that for any graph
  $(G,\mu)$ satisfying $CD(0,\infty)$ and any natural number $k\geq 2$
  we have
  \begin{equation}\label{eq:isoperimetricratioIntroduction}
   h_k(G,\mu)\leq C D_G^{non}D_G^{nor}k\sqrt{\log k}\,h_2(G,\mu).
 \end{equation}
\end{corolet}

Higher order Cheeger constants lead naturally to the notion of
\emph{$k$-way expanders} introduced by Tanaka \cite{Tanaka13} and
Mimura \cite{Mimura2014} (where $2$-way expander families coincide
with the classical families of expanders). The condition of being a
$k$-way expander family is strictly weaker than the property of being
a classical expander family (see \cite[p. 2525]{Mimura2014}). A
consequence of Corollary \ref{cor:isoperimetricratioIntroduction} is
the fact that the concepts of $k$-way expanders for all $k \ge 2$ are
equivalent in the class of all graphs satisfying $CD(0,\infty)$. This
can be viewed as an analogue (for the $CD(0,\infty)$-class) to
Mimura's result \cite[Corollary 1.5]{Mimura2014} for the class of all
vertex-transitive graphs.

\subsection{Organisation of the paper}

In Section \ref{section:curvinfo}, we discuss in detail the
curvature-dimension inequality in the graph setting, introduce two
interesting examples for later use concerning the optimality of
Theorem \ref{thm:eigenvalueratioIntroduction}, and provide a proof of
Theorem \ref{thm:CartesianProductIntroduction}. In Section
\ref{section:eigenratio}, we derive the eigenvalue ratio estimate,
discuss its optimality with the help of examples, and present
applications.  In Section \ref{section:isoratio}, we discuss a higher
order Cheeger constant ratio estimate and related topics about
multi-way expanders. Finally in the Appendix, we give more details
about the curvature dimension inequality calculations in some examples
and also a self-contained proof of Buser's inequality for graphs
satisfying $CD(0,\infty)$.

\section{Information for a better understanding of curvature}
\label{section:curvinfo}

The curvature-dimension inequality (CD-inequality for short) was
introduced by Bakry and \'{E}mery \cite{BaEm} as a substitute of the lower
Ricci curvature bound of the underlying space. It was studied on
graphs by Schmuckenschl\"ager \cite{Schmuckenschlaeger} and Lin and Yau
\cite{LinYau}, see also \cite{JostLiu14}, \cite{ChungLinYau14}. The
operators $\Gamma$ and $\Gamma_2$ are defined iteratively as follows.

\begin{definition}
  For any two functions $f,g: V\rightarrow \mathbb{R}$, we define
  \begin{equation}\label{eq:Gamma}
    \Gamma(f,g):=\frac{1}{2}\{\Delta(fg)-f\Delta g-g\Delta f\},
  \end{equation}
  and
  \begin{equation}\label{eq:Gammatwo}
    \Gamma_2(f,g):=\frac{1}{2}\{\Delta\Gamma(f,g)-\Gamma(f, \Delta g)-
    \Gamma(g, \Delta f)\}.
  \end{equation}
\end{definition}

We also write $\Gamma(f):=\Gamma(f,f)$ and
$\Gamma_2(f):=\Gamma_2(f,f)$ for short. In particular, by the
definition above we have for any $x\in V$ and any $f,g$
\begin{equation}\label{eq:gradient}
  \Gamma(f,g)(x)=\frac{1}{2\mu(x)}\sum_{y,y\sim x}w_{xy}(f(y)-f(x))(g(y)-g(x)).
\end{equation}
A useful fact is the following summation by part formula,
\begin{equation}\label{eq:summaBypart}
\sum_{x\in V}\mu(x)\Gamma(f,g)(x)=-\sum_{x\in V}\mu(x)f(x)\Delta g(x),
\end{equation}
and also
\begin{equation}\label{eq:gammaHolder}
\Gamma(f,g)\leq\sqrt{\Gamma(f)}\sqrt{\Gamma(g)}.
\end{equation}
Rewriting (\ref{eq:Gamma}) provides the following chain rule,
\begin{equation}
  \Delta(f^2)=2\Gamma(f)+2f\Delta(f).
\end{equation}

\begin{definition}\label{defn:CDinequality}
  Let $K \in \mathbb {R}$ and $n \in {\mathbb R}_+$. We say
  that $(G,\mu)$ satisfies the CD-inequality $CD(K,n)$ if for any
  functions $f$ and any vertex $x$, the following inequality holds,
  \begin{equation}
    \Gamma_2(f)(x)\geq \frac{1}{n}(\Delta f(x))^2+K\Gamma(f)(x).
  \end{equation}
\end{definition}

In particular, we say that $(G,\mu)$ satisfies $CD(0,\infty)$ if for
any functions $f$ we have $\Gamma_2(f)\geq 0$.

In the following subsection we present some illustrative examples of
weighted graphs and their curvature. Subsection 2.2 describes a method
to construct many more examples satisfying $CD(0,\infty)$ from given
ones via Cartesian products and provides a proof of Theorem
\ref{thm:CartesianProductIntroduction}.

\subsection{Examples of graphs satisfying $CD(0,\infty)$}

We will be mainly concerned with the class of graphs satisfying
$CD(0,\infty)$. For purpose of illustration and for later use
concerning the optimality of Theorem
\ref{thm:eigenvalueratioIntroduction}, we present some simple
examples. While explicit curvature calculations of these examples are
given in Appendix \ref{section:appendixTT}, we first mention some
basic principles used in our curvature calculations.

From (\ref{eq:Gammatwo}), we see that $\Gamma_2$ is a symmetric
bilinear form. At every vertex $x\in V$, we can write
$\Gamma_2(f,g)(x) = f^\top \Gamma_2(x) g$, where
$f, g \in \mathbb{R}^{V}$ on the right hand side denote the (column)
vector representation of the functions $f$ and $g$.  Let
$B_2(x):=\{y\in V: \text{dist}(y,x)\leq 2\}$, where $\text{dist}$
stands for the usual shortest-path metric on $V$. Then $\Gamma_2(x)$
is a symmetric matrix, which is non-trivial only on a submatrix of
size $|B_2(x)|\times |B_2(x)|$ which we, denote, again, by
$\Gamma_2(x)$, for simplicity. A graph $(G,\mu)$ satisfies
$CD(0,\infty)$ if and only if $\Gamma_2(x)$ is positive-semidefinite
at every vertex $x\in V$. Observe that the entries of each row of $\Gamma_2(x)$ sum up to zero since $\Gamma_2(x)\mathbf{c}=0$ for any constant
vector $\mathbf{c}$. In particular, if all the diagonal entries are
nonnegative and all the off-diagonal entries are nonpositive, then the
matrix $\Gamma_2(x)$ is diagonally dominant and hence
positive-semidefinite.

\begin{example}\label{example:triangle}
  Consider the triangle graph $\triangle_{xyz}$ with positive edge
  weights $a,b,c$, as shown in Figure \ref{F1}. Assign a measure $\mu$
  to the vertices such that $\mu(x):=C, \mu(y):=B, \mu(z):=A$.
  \begin{itemize}
  \item \emph{Normalized case}, that is, $A=b+c, B=a+c, C=a+b$. Then
    $(\triangle_{xyz},\mu)$ satisfies $CD(0,\infty)$.
  \item \emph{Non-normalized case}, that is, $A=B=C=1$. Then
    $(\triangle_{xyz},\mu)$ does not always satisfy $CD(0,\infty)$. If
    in particular $a=c$, it satisfies $CD(0,\infty)$. But when $a=1,
    c=1/b$, it does not satisfy $CD(0,\infty)$ if $b$ is large/small
    enough. In fact, when $b\geq 5.01$ or $b\leq 0.12$, the symmetric
    curvature matrix $\Gamma_2(x)$ has a negative eigenvalue.
\end{itemize}

This example illustrates the general observation that positivity
of the non-normalized Bakry-\'Emery curvature at a vertex is more
sensitive to large differences in the weights of the adjacent edges than
normalized Bakry-\'Emery curvature.
\end{example}

\begin{figure}[h]
\begin{minipage}[t]{0.45\linewidth}
\centering
\includegraphics[width=0.65\textwidth]{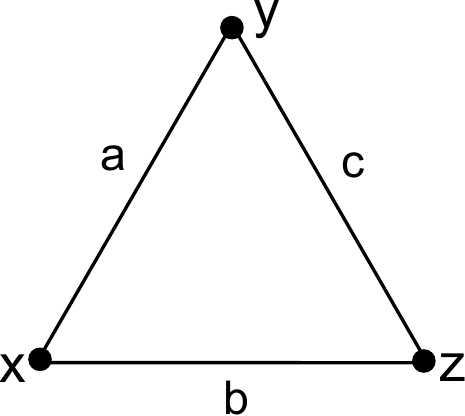}
\caption{Triangle\label{F1}}
\end{minipage}
\hfill
\begin{minipage}[t]{0.45\linewidth}
\centering
\includegraphics[width=0.7\textwidth]{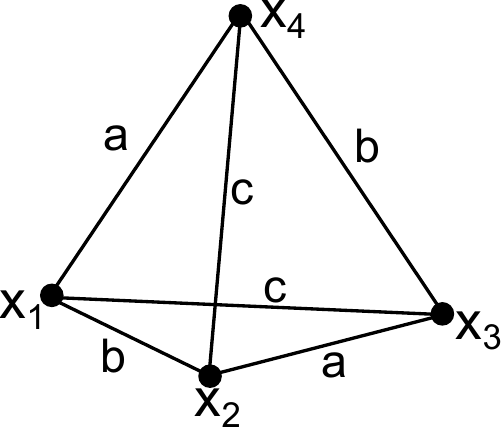}
\caption{Tetrahedron\label{F2}}
\end{minipage}
\end{figure}

\begin{example}\label{example:tetrahedron}
  Consider the tetrahedron graph $T_4$ with positive edge weights
  $a,b,c$ as shown in Figure \ref{F2}. Observe that this graph is
  regular, i.e., $d_{x_i}=a+b+c$ is a constant for every $i$. Assign a
  measure $\mu$ on the vertices such that $\mu(x_i)=A$ for all $i$,
  where $A$ is a positive constant. (Note that this includes both the
  cases of normalized and non-normalized Laplacians.) Then $(T_4,\mu)$
  always satisfies $CD(0,\infty)$.
\end{example}

The details about the curvature matrix $\Gamma_2$ of the triangle
graph and the tetrahedron graph are given in Appendix
\ref{section:appendixTT}. For the normalized case, the curvature of
unweighted triangle graphs was calculated in
\cite[Proposition 1.6]{LinYau}, and the curvature of general unweighted complete graphs
was calculated in \cite[Proposition 3]{JostLiu14}.

In fact, the tetrahedron graph in Figure \ref{F2} belongs to a large
class of graphs called Ricci flat graphs with consistent edge
weights. The concept of a Ricci flat graph was introduced by Chung and
Yau \cite{ChungYau96} and that of consistent edge weights was further
introduced in Bauer, Horn, Lin, Lippner, Mangoubi and Yau
\cite{BHLLMY13}. We refer the reader to \cite{ChungYau96,BHLLMY13} for
the precise definitions. Every graph in this class is a regular graph
(in fact both its unweighted and weighted degree are constant) and
satisfies $CD(0,\infty)$ if we assign a measure $\mu$ such that
$\mu(x)=A$ for all vertices $x$ (see \cite{LinYau, ChungYau96} for the
unweighted case, the weighted case follows from the same
calculations). In particular, every abelian Cayley graph is Ricci flat
and hence satisfies $CD(0,\infty)$.




\subsection{CD-inequalities of Cartesian product graphs}

In this subsection we discuss a method for constructing new graphs
satisfying certain CD-inequalities from known examples, that is,
taking the Cartesian product.

Given two (possibly infinite) graphs $G_1=(V_1, E_1, w)$ and
$G_2=(V_2, E_2, \overline{w})$, their Cartesian product $G_1\times
G_2=(V_1\times V_2, E_{12}, w^{12})$ is a weighted graph with vertex
set $V_1\times V_2$ and edge set $E_{12}$ given by the following
rule. Two vertices $(x_1,y_1), (x_2, y_2)\in V_1\times V_2$ are
connected by an edge in $E_{12}$ if $$x_1=x_2, y_1\sim y_2\text{ in
}E_2\,\,\,\text{ or }\,\,\,x_1\sim x_2 \text{ in }E_1, y_1=y_2.$$
In the first case above we chose the edge weight to be
$\overline{w}_{y_1y_2}$ and in the second case $w_{x_1x_2}$.

Recall the following result from the Introduction which we will prove
in this subsection.

\begin{thmB}
  If $(G_1, \mathbf{1}_{V_1})$ and $(G_2, \mathbf{1}_{V_2})$ satisfy
  $CD(K_1, n_1)$ and $CD(K_2, n_2)$, respectively, then $(G_1\times
  G_2, \mathbf{1}_{V_1\times V_2})$ satisfies $CD(K_1\wedge K_2,
  n_1+n_2)$.
\end{thmB}


Let $f:V_1\times V_2\rightarrow\mathbb{R}$ be a function on the
product graph. For fixed $y\in V_2$, we will write
$f_y(\cdot):=f(\cdot,y)$ as a function on $V_1$. Similarly,
$f^x(\cdot):=f(x,\cdot)$.  The following lemma is crucial for the
proof.

\begin{lemma}\label{lemma:cartesianproduct}
  For any function $f:V_1\times V_2\rightarrow\mathbb{R}$ and any
  $(x,y)\in V_1\times V_2$, we have
  \begin{equation}
    \Gamma_2(f)(x,y)\geq \Gamma_2(f_y)(x)+\Gamma_2(f^x)(y),
  \end{equation}
  where the operators $\Gamma_2$ are understood to be on different
  graphs according to the functions they are acting on.
\end{lemma}

\begin{proof}
  For simplicity, we will denote by $x_i$ a neighbor of $x\in V_1$,
  and write shortly $w_i:=w_{xx_i}$ in this proof. Similar notions are
  used for $y\in V_2$ and $\overline{w}$.

  Recall $2\Gamma_2(f)(x,y)=\Delta\Gamma(f)(x,y)-2\Gamma(f,\Delta
  f)(x,y)$. By definition, we have
  \begin{align*}
    \Delta\Gamma(f)(x,y)=&\sum_{x_i\sim
      x}w_i\left(\Gamma(f)(x_i,y)-\Gamma(f)(x,y)\right)
    \\&+\sum_{y_k\sim
      y}\overline{w}_k\left(\Gamma(f)(x,y_k)-\Gamma(f)(x,y)\right):=L_1+L_2.
  \end{align*}
  For the first term $L_1$, we calculate
  \begin{align*}
    L_1=&\sum_{x_i\sim x}w_i\left[\Gamma(f_y)(x_i)+\Gamma(f^{x_i})(y)-\Gamma(f_y)(x)-\Gamma(f^x)(y)\right]\\
    =&\Delta\Gamma(f_y)(x)+\frac{1}{2}\sum_{x_i\sim x}\sum_{y_k\sim y}w_i\overline{w}_k\big[(f(x_i,y_k)-f(x_i,y))^2\\
    &\hspace{5cm}-(f(x,y_k)-f(x,y))^2\big].
  \end{align*}
  Similarly, we obtain
  \begin{align*}
    L_2=&\Delta\Gamma(f^x)(y)+\frac{1}{2}\sum_{y_k\sim y}\sum_{x_i\sim
      x}\overline{w}_kw_i\big[(f(x_i,y_k)-f(x,y_k))^2
    \\&\hspace{5cm}-(f(x_i,y)-f(x,y))^2\big].
  \end{align*}
  Furthermore, we have
  \begin{align*}
    2\Gamma(f,\Delta f)(x,y)=&\sum_{x_i\sim x}w_i(f(x_i, y)-f(x,y))\left(\Delta f(x_i,y)-\Delta f(x,y)\right)+\\
    &\sum_{y_k\sim y}\overline{w}_k(f(x, y_k)-f(x,y))\left(\Delta f(x,y_k)-\Delta f(x,y)\right)\\
    :=&T_1+T_2.
  \end{align*}
  Then for the term $T_1$, we have
  \begin{align*}
    T_1=&\sum_{x_i\sim x}w_i(f(x_i, y)-f(x,y))\left(\Delta f_y(x_i)+\Delta f^{x_i}(y)-\Delta f_y(x)-\Delta f^x(y)\right)\\
    =&2\Gamma(f_y,\Delta f_y)(x)+\sum_{x_i\sim x}\sum_{y_k\sim
      y}w_i\overline{w}_k(f(x_i,y)-f(x,y))\times
    \\&\hspace{4cm}(f(x_i, y_k)-f(x_i,y)-f(x,y_k)+f(x,y)).
  \end{align*}
  Similarly, we also have
  \begin{align*}
    T_2=&2\Gamma(f^x,\Delta f^x)(y)+\sum_{y_k\sim y}\sum_{x_i\sim
      x}\overline{w}_kw_i(f(x,y_k)-f(x,y))\times
    \\&\hspace{4cm}(f(x_i, y_k)-f(x,y_k)-f(x_i,y)+f(x,y)).
  \end{align*}
  Observing the fact that
  \begin{align*}
    &(f(x_i,y_k)-f(x,y_k))^2-(f(x_i,y)-f(x,y))^2\\
    =&(f(x_i,y_k)-f(x,y_k)-f(x_i,y)+f(x,y))^2+\\
    &2(f(x_i,y_k)-f(x,y_k)-f(x_i,y)+f(x,y))(f(x_i,y)-f(x,y)),
  \end{align*}
  We arrive at
  \begin{align}\label{eq:halflemm1}
    L_2-\Delta\Gamma(f^x)(y)-(T_1-2\Gamma(f_y,\Delta f_y)(x))\geq 0,
  \end{align}
  and
  \begin{align}\label{eq:halflemm2}
    L_1-\Delta\Gamma(f_y)(x)-(T_2-2\Gamma(f^x,\Delta f^x)(y))\geq 0.
  \end{align}
  This completes the proof.
\end{proof}

\begin{remark}
  The intuition of the above calculation is that the mixed terms are
  "flat". In fact, Lemma \ref{lemma:cartesianproduct} still holds if
  we replace $\Gamma_2(f)$ by $\widetilde{\Gamma}_2(f):=
  \frac{1}{2}\Delta\Gamma(f)-\Gamma\left(f,\frac{\Delta(f^2)}{2f}\right)$.
  Explicitly, for any positive function $f:V_1\times V_2\rightarrow
  \mathbb{R}$ and any $(x,y)\in V_1\times V_2$, we have
  \begin{equation}
    \widetilde{\Gamma}_2(f)(x,y)\geq \widetilde{\Gamma}_2(f_y)(x)+\widetilde{\Gamma}_2(f^x)(y).
  \end{equation}
  The proof is done in a similar way. The operator
  $\widetilde{\Gamma}_2$ was introduced in \cite{BHLLMY13} to define a
  modification of the CD-inequality, called exponential
  curvature-dimension inequality $CDE(K, n)$ (see Definition 3.9 in
  \cite{BHLLMY13}). Under the assumption of their new notion of
  curvature lower bound, they prove Li-Yau type gradient estimates
  (dimension-dependent) for the heat kernels on graphs.
\end{remark}

\begin{proof}[Proof of Theorem \ref{thm:CartesianProductIntroduction}]
  By Lemma \ref{lemma:cartesianproduct}, we have for any function $f:
  V_1\times V_2\rightarrow \mathbb{R}$ and any $(x,y)\in V_1\times
  V_2$,
  \begin{align}
    &\Gamma_2(f)(x,y)\geq \Gamma_2(f_y)(x)+\Gamma_2(f^x)(y)\notag\\
    \geq &\frac{1}{n_1}(\Delta f_y(x))^2+\frac{1}{n_2}(\Delta f^x(y))^2+K_1\Gamma(f_y)(x)+K_2\Gamma(f^x)(y)\notag\\
    \geq &\frac{1}{n_1+n_2}(\Delta f_y(x)+\Delta f^x(y))^2+K_1\wedge
    K_2(\Gamma(f_y)(x)+\Gamma(f^x)(y)).\label{eq:dimensionestiamte}
  \end{align}
  In the last inequality above we used Young's inequality. Recalling
  the facts $\Delta f_y(x)+\Delta f^x(y)=\Delta f(x,y)$ and
  $\Gamma(f_y)(x)+\Gamma(f^x)(y)=\Gamma (f)(x,y)$, we complete the
  proof.
\end{proof}

\begin{remark}\label{rmk:extendCartesian}
  We can have more flexibility concerning the measures assigned to
  vertices. Suppose the vertex measures assigned to $G_1, G_2$ and
  $G_1\times G_2$ take the constant values $\mu_1, \mu_2$ and
  $\mu_{12}$ on each vertex, respectively, then the modified
  conclusion of Theorem \ref{thm:CartesianProductIntroduction} is that
  $(G_1\times G_2, \mu_{12})$ satisfies
  \begin{equation}\label{eq:CartesianGeneral} CD\left(\frac{1}{\mu_{12}}
      (\mu_1 K_1\wedge\mu_2 K_2),n_1+n_2\right).
  \end{equation}
  This modification covers the case of normalized Laplacians on
  regular graphs. In particular, if both $(G_1, \mu_1)$ and $(G_2,
  \mu_2)$ satisfy $CD(0, \infty)$, then $(G_1\times G_2, \mu_{12})$
  also satisfies $CD(0,\infty)$.
\end{remark}

\begin{remark}\label{rmk:optimalCartesian}
  The estimates of the CD-inequality in Theorem
  \ref{thm:CartesianProductIntroduction} (in fact also
  (\ref{eq:CartesianGeneral})) are tight at least for the Cartesian
  product of a graph $G$ with itself. That is, if $G$ satisfies
  $CD(K, n)$ precisely (i.e., for given dimension $n$, $K$ is chosen
  largest possible), then the CD-inequality in Theorem
  \ref{thm:CartesianProductIntroduction} (or in
  (\ref{eq:CartesianGeneral})) is optimal for $G\times G$. This can be
  seen as follows. First note that this tightness depends on that of
  (\ref{eq:halflemm1}), (\ref{eq:halflemm2}) and
  (\ref{eq:dimensionestiamte}). By assumption, there exists a function
  $f$ on the graph $G$ and a vertex $x$ of the graph such that
  \begin{equation*}
    \Gamma_2(f)(x)=\frac{1}{n}(\Delta f(x))^2+K\Gamma(f)(x),
  \end{equation*}
  with $\Gamma(f)(x)\neq 0$. We can then choose a particular function
  $F$ on $G\times G$ such that (\textrm{i}) $F(x,x)=f(x)$;
  (\textrm{ii})$F(x_i,x)=f(x_i)$ for all neighbors $x_i$ of $x$ in
  $G$; (\textrm{iii}) $F(x,x_k)=f(x_k)$ for all neighbors $x_k$ of $x$
  in $G$; (\textrm{iv}) $F(x_i,x_k)=F(x,x_k)+F(x_i,x)-F(x,x)$. For
  such a $F$ the equalities in (\ref{eq:halflemm1}) and
  (\ref{eq:halflemm2}) are attained at $(x,x)$ and $\Delta
  F_x(x)=\Delta F^x(x)$ and (by consequence) $\Gamma(F_x)(x)=\Gamma(F^x)(x)$, hence the equality in
  (\ref{eq:dimensionestiamte}) is also attained. Therefore we obtain
  \begin{equation*}
    \Gamma_2(F)(x,x)=\frac{1}{2n}(\Delta F(x,x))^2+K\Gamma(F)(x,x),
  \end{equation*}
  which confirms the postulated tightness.

  In the specific example $(G,\mathbf{1}_V)$, where $G$ is the
  unweighted graph consisting of just one edge with end-points $x,y$,
  and writing $\Gamma(f,g)(x) = f^\top \Gamma(x) g$ and $\Delta f(x) =
  \Delta(x) f$, an easy calculation leads to
  $$ \Gamma_2(x) = 2 \Gamma(x) = \Delta(x)^\top \Delta(x) = \begin{pmatrix}
    1 & -1 \\ -1 & 1 \end{pmatrix}, $$
  and the curvature-dimension condition $CD(K,n)$ translates into
  $K \le 2 - 2/n$. This means that $G$ satisfies $CD(2-2/n,n)$ precisely
  and $G \times G$ satisfies $CD(2-2/n,2n)$ precisely, as well.
\end{remark}

\section{Eigenvalue ratios and higher order spectral
  bounds}\label{section:eigenratio}

The first subsection is concerned with the eigenvalue ratio estimate
under the $CD(0,\infty)$ condition and its optimality
properties. Subsections 3.2 and 3.3 discuss applications: lower
estimates for higher order Cheeger constants and upper diameter
estimates in term of eigenvalues.

\subsection{Eigenvalue ratio}

As in \cite{Liu14} for the Riemannian manifold case, we need to
combine the improved Cheeger inequality with the following Buser type
inequality:

\begin{thm}\label{thm:Buser}
  Let $(G,\mu)$ satisfy $CD(0,\infty)$. Then we
  have
  \begin{equation}\label{eq:BuserCD0}
    h_2(G,\mu)\geq \frac{e-1}{2e}\frac{1}{\sqrt{D_G^{nor}}}
    \sqrt{\lambda_2(G,\mu)}.
  \end{equation}
\end{thm}

This is an adaption of the Buser inequality in \cite{KlarKozma} to our
setting of weighted graphs with a slightly better constant in
\eqref{eq:BuserCD0}. For the reader's convenience, we present a proof
for (\ref{eq:BuserCD0}) in Appendix \ref{section:appendixBuser}. The
dependence on $D_G^{nor}$ comes from Lemma \ref{lemma:boundaryMeasure}
there, see also \cite{BHLLMY13}.

We also need the following improved Cheeger inequality in
\cite{KLLGT2013} to obtain the eigenvalue ratio estimate. Their
context was the weighted normalized setting.

\begin{thm}[Kwok, Lau, Lee, Oveis Gharan and Trevisan]
  On $(G,\mu)$ we have for any natural number $k \ge 2$,
  \begin{equation}\label{eq:improveCheeger}
    h_2(G,\mu)\leq 10\sqrt{2D_G^{non}}k\frac{\lambda_2(G,\mu)}{\sqrt{\lambda_k(G,\mu)}}.
  \end{equation}
\end{thm}

Here the setting is slightly more general than that in
\cite{KLLGT2013}. To obtain (\ref{eq:improveCheeger}), one needs to be
careful about the final calculations in the proof of Proposition 3.2
in \cite{KLLGT2013} (pp.16 in the full version of \cite{KLLGT2013})
and the fact that $\lambda_k\leq 2D_G^{non}$.

Combining (\ref{eq:BuserCD0}) and (\ref{eq:improveCheeger}), we get
the following eigenvalue ratio estimate stated in the Introduction:

\begin{thmA}
  For any finite graph $(G,\mu)$ satisfying
  $CD(0,\infty)$ and any natural number $k \ge 2$, we have
  \begin{equation} \label{eq:eigenvalueratio}
  \lambda_k(G,\mu) \le \left(\frac{20\sqrt{2}e}{e-1}\right)^2
    D_G^{non}D_G^{nor}k^2\lambda_2(G,\mu).
  \end{equation}
\end{thmA}

We remark that this estimate does not depend on the size of the graph.
The following examples are concerned with the optimality of this result.

The first example shows that the order of $k$ in the above
estimate is optimal.

\begin{example}\label{example:Cycle}
  Consider an unweighted cycle $\mathcal{C}_N$ with $N\geq 3$
  vertices. Note that $\mathcal{C}_N$ can be considered as an abelian
  Cayley graph and hence satisfies $CD(0,\infty)$. Assign to it a
  measure $\mu$ which takes the constant value $2$ on every
  vertex. Then the eigenvalues of the associated Laplacian are given
  by (see e.g. Example 1.5 in \cite{Chung} or Section 7 in
  \cite{Liu13}),
  \begin{equation*}
    \lambda_k(\mathcal{C}_N)=1-\cos\left(\frac{2\pi}{N}\left\lfloor\frac{k}{2}\right\rfloor\right), \,\,\,k=1,2,\ldots,N.
  \end{equation*}
  Observe that we have
  \begin{equation*}
    \lim_{N\rightarrow \infty}\frac{\lambda_k(\mathcal{C}_N)}{\lambda_2(\mathcal{C}_N)}=\left\lfloor\frac{k}{2}\right\rfloor^2.
  \end{equation*}
\end{example}

The dependence on the term $D_G^{non}D_G^{nor}$ is also necessary in
the estimate (\ref{eq:eigenvalueratio}). This can be concluded from
the following examples.

\begin{example}\label{example:triangle2}
  Let us revisit the triangle graph $(\triangle_{xyz},\mu)$ in Example
  \ref{example:triangle}. Consider the special case that $A=B=C=1$ and
  $a=c$. Suppose $b\geq a$. Then this graph satisfies
  $CD(0,\infty)$. The eigenvalues of the non-normalized Laplacian are
  \begin{equation*}
    \lambda_1 = 0< \lambda_2 = 3a \leq \lambda_3 = a+2b.
  \end{equation*}
  Note further that $D_G^{non}D_G^{nor}=(a+b)/a$. Therefore, we have
  \begin{equation}
    \frac{1}{3}D_G^{non}D_G^{nor}\leq\frac{\lambda_3(\triangle_{xyz})}{\lambda_2(\triangle_{xyz})}\leq\frac{2}{3}D_G^{non}D_G^{nor}.
  \end{equation}
\end{example}

We give another example which works for the eigenvalue ratios of both
non-normalized and normalized Laplacians.

\begin{example}\label{example:tetrahedron2}
  Consider the tetrahedron graph $(T_4,\mu)$ in Example
  \ref{example:tetrahedron} with the assumption that $b\geq
  a=c$. Recall $\mu=A$ is a constant measure. Then the eigenvalues of
  the $\mu$-Laplacian are
  \begin{equation*}
    \lambda_1 = 0 < \lambda_2 = \frac{4a}{A} \leq \lambda_3 = \lambda_4 =
    \frac{2a+2b}{A}.
  \end{equation*}
  Moreover, we have $D_G^{non}D_G^{nor} = (2a+b)/a$. Hence we obtain
  \begin{equation}
    \frac{1}{4}D_G^{non}D_G^{nor}\leq\frac{\lambda_3(T_4)}{\lambda_{2}(T_4)}\leq\frac{1}{2}D_G^{non}D_G^{nor}.
  \end{equation}
\end{example}

The following example shows that we cannot expect that the eigenvalue
ratio estimate (\ref{eq:eigenvalueratio}) remains valid if a graph
$(G,\mu)$ possesses a small portion of vertices not satisfying
$CD(0,\infty)$.

\begin{example}\label{example:dumbbell}
  Consider a sequence of dumbbell graphs
  $\{G_N\}_{N=3}^{\infty}$. Given two copies of complete graphs over
  $N$ vertices, $\mathcal{K}_N$ and $\mathcal{K}_N'$, $G_N$ is the
  graph obtained via connecting them by a new edge $e=(y_0, y'_0)$ as
  shown in Figure \ref{F3}.
  \begin{figure}[h]
    \centering
    \includegraphics[width=0.5\textwidth]{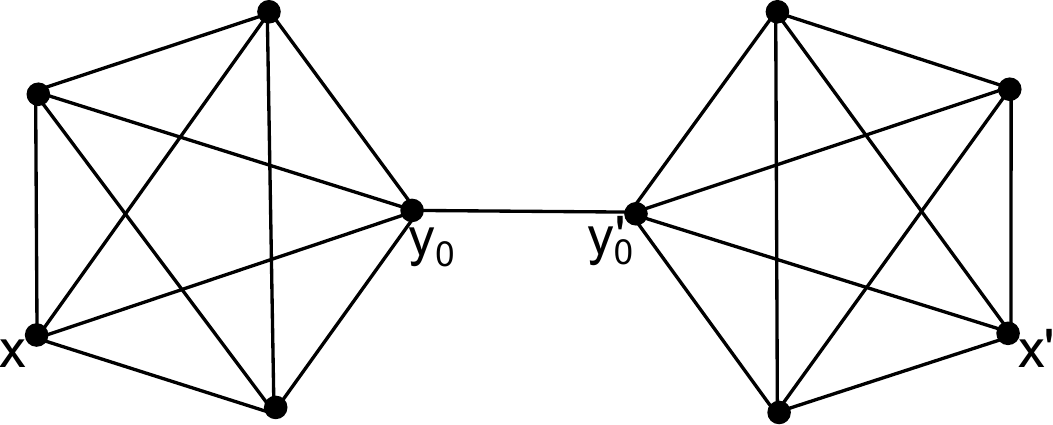}
    \caption{The dumbbell graph $G_5$\label{F3}}
  \end{figure}
  It was shown in \cite{JostLiu14} that the complete graph
  $\mathcal{K}_N$ with normalized Laplacian satisfies
  $CD\left(\frac{N+2}{2(N-1)}, \infty\right)$. Modifying the
  calculation in the proof of this fact in \cite{JostLiu14}, we obtain
  the following results.
  \begin{itemize}
  \item With the normalized Laplacian, $G_N$ satisfies
    $CD\left(\frac{1}{2},\infty\right)$ at every vertex which is not
    $y_0,y_0'$. At $y_0, y_0'$, $CD(0,\infty)$ does not hold when
    $N\geq 3$.
  \item With the non-normalized Laplacian, $G_N$ satisfies
    $CD\left(\frac{N}{2},\infty\right)$ at every vertex which is not
    $y_0,y_0'$. At $y_0, y_0'$, $CD(0,\infty)$ does not hold when
    $N\geq 3$.
  \end{itemize}
  We present the calculations in Appendix
  \ref{section:appendixDumbell}. With only $2$ of $2N$ vertices
  violating the curvature condition, the eigenvalue ratio estimate
  (\ref{eq:eigenvalueratio}) does not hold any more. Indeed, for the
  normalized Laplacian, we observe by Cheeger's inequality that
  $$ \lambda_2(G_N)\leq 2h(G) \le2\frac{|E(S,V\backslash S)|}{\mu(S)} = \frac{2}{N(N-1)+1}, $$
  choosing $S$ equals $\mathcal{K}_N$ to estimate $h(G)$, given in \eqref{eq:classcheeg}.

  Recall that the spectrum of a complete graph $\mathcal{K}_N$ is the
  simple eigenvalue $0$ and the eigenvalue $\frac{N}{N-1}$ with
  multiplicity $N-1$. Deleting the edge $\{y_0,y_0'\}$ from $G_N$, we
  obtain two disjoint copies of $\mathcal{K}_N$ with combined spectrum
  $\lambda_1=\lambda_2=0 < \lambda_3=\lambda_4 = \cdots =
  \frac{N}{N-1}$.  By an interlacing theorem for edge-deleting in
  \cite{Chenetal2004}, we conclude that
  $\lambda_4(G_N)\geq \frac{N}{N-1}$. (Note that the Laplacian
  $\mathcal{L}$ there is slightly different but unitarily equivalent
  to our normalized Laplacian, since
  $\mathcal{L} = D^{1/2} \Delta D^{-1/2}$.)  Therefore we have
  \begin{equation*}
    \frac{\lambda_4(G_N)}{\lambda_2(G_N)}\geq \frac{1}{2}N^2.
  \end{equation*}
  Since in this case $D_G^{non}D_G^{nor}=N$,
  (\ref{eq:eigenvalueratio}) does not hold when $N$ is large. Similar
  arguments show also for the non-normalized Laplacian that
  (\ref{eq:eigenvalueratio}) is no longer true for all $N$. (The
  interlacing theorem for non-normalized Laplacian is well-known, see
  e.g. \cite{Heuvel95}).
\end{example}

\begin{remark}
  Replacing the sequence of complete graphs $\mathcal{K_N}$ above by a
  sequence of expanders, we obtain graphs \emph{of bounded degree}
  violating the curvature condition and for which
  (\ref{eq:eigenvalueratio}) does not hold.
\end{remark}

\subsection{Higher order Buser inequalities}

Higher order Buser inequalities were first established by Funano
\cite{Funano2013} in the Riemannian setting and then improved in
\cite{Liu14}. The following result from the Introduction seems to be
the first higher order Buser type inequality in the graph setting.

\begin{corC}
  For any graph $(G, \mu)$ satisfying $CD(0,\infty)$ and any
  natural number $k$, we have
  \begin{equation} \label{eq:higherorderBuser}
    h_k(G,\mu)\geq h_2(G,\mu)\geq \frac{(e-1)^2}{40\sqrt{2}e^2}\frac{1}{D_G^{nor}\sqrt{D_G^{non}}}\frac{\sqrt{\lambda_k(G,\mu)}}{k}.
  \end{equation}
\end{corC}

\begin{proof}
  The first inequality is given by the monotonicity of the higher
  order Cheeger constants $h_k(G,\mu)$ (as functions in $k$). The
  second inequality follows from Buser's inequality
  (\ref{eq:BuserCD0}) and Theorem
  \ref{thm:eigenvalueratioIntroduction}.
\end{proof}

\begin{remark} Inequalities in the other direction complementing
  \eqref{eq:higherorderBuser} (without any curvature condition) are
  given by the higher order Cheeger inequalities
  \eqref{eq:highordcheeg} by Lee, Oveis Gharan and Trevisan \cite{LGT2013}
  from the Introduction. In our setting of weighted graphs, they read
  as
  \begin{equation}
    h_k(G,\mu)\leq C\sqrt{D_G^{non}}k^2\sqrt{\lambda_k(G,\mu)},
  \end{equation}
  where $C$ is an universal constant. (For the generalization into our
  setting, one needs to slightly modify the calculation for
  $\mathbb{E}\left(\sum_{i=1}^mw(E(\hat{S}_i,\overline{\hat{S}}_i))\right)$
  in Lemma 4.7 of \cite{LGT2013}.)  Hence, for a graph $(G,\mu)$
  satisfying $CD(0,\infty)$ and with bounded degree, $h_k(G,\mu)$ and
  $\sqrt{\lambda_k(G,\mu)}$ are equivalent up to polynomials of $k$ of degree smaller or equal to $2$.
\end{remark}

\begin{remark}
  In \cite{BHLLMY13}, Bauer, Horn, Lin, Lippner, Mangoubi and Yau
  proved for a graph $(G,\mu)$ satisfying another, related curvature
  condition, namely, the \emph{exponential curvature-dimension
    inequality $CDE(0,n)$} (see \cite[Definition 3.9]{BHLLMY13}) and
  for a fixed $0<\alpha<1$ that there exists a constant $C(\alpha)$,
  depending only on $\alpha$, such that
  \begin{equation}
    \lambda_2(G,\mu)\leq C(\alpha)D_G^{nor}n h_2(G,\mu)^2.
  \end{equation}
  That is, they obtain a dimension-dependent Buser inequality. Our
  approach also applies to their setting. In particular, we obtain the
  following eigenvalue ratio estimate and higher order Buser
  inequalities under the condition $CDE(0,n)$,
  \begin{equation}
    \lambda_k(G,\mu)\leq C_1(\alpha)D_G^{nor}D_G^{non}nk^2\lambda_2(G,\mu),
  \end{equation}
  \begin{equation}
    h_k(G,\mu)\geq h_2(G,\mu)\geq C_2(\alpha)\frac{1}{D_G^{nor}\sqrt{D_G^{non}}}\frac{1}{nk}\sqrt{\lambda_k(G,\mu)},
  \end{equation}
  where $C_1(\alpha)$, $C_2(\alpha)$ are constants depending only on
  $\alpha$.
\end{remark}

\subsection{A discrete analogue of Cheng's Theorem}

In the manifold setting, Cheng's Theorem \cite{Cheng75} provides a
relation between the diameter and the $k$-eigenvalue of the Laplacian
under non-negative Ricci curvature assumption, presented in
\eqref{eq:cheng} in the Introduction. In this subsection, we derive a
graph theoretical analogue. To do so, we restrict our considerations
to Alon and Milman's setting \cite{AM1985} of unweighted non-normalized
graphs $(G,\mu)$ with $\mu = \mathbf{1}_V$. We recall the following
eigenvalue-diameter estimate from \cite[Theorem 2.7]{AM1985}.

\begin{thm}[Alon and Milman] \label{thm:alonmilman}
  Let $G=(V,E)$ be a finite connected graph with maximal degree $d_G$
  and $\Delta$ be the non-normalized Laplacian. Then we have
  \begin{equation}\label{eq:AMdiameter}
{\rm diam}(G) \le 2 \sqrt{\frac{2d_G}{\lambda_2(G)}} \log_2 |V|.
\end{equation}
\end{thm}

Combining Theorem \ref{thm:alonmilman} with Theorem
\ref{thm:eigenvalueratioIntroduction}, we obtain the following result
from the Introduction.

\begin{corD} Let $(G,\mathbf{1}_V)$ be an unweighted finite graph
  satisfying $CD(0,\infty)$. Then we have for any $k \ge 2$:
  \begin{equation}\label{eq:AMlambdak}
    {\rm diam}(G)\leq \frac{80 e}{e-1} d_G \log_2|V| \frac{k}{\sqrt{\lambda_k(G,\mu)}}.
  \end{equation}
\end{corD}

\begin{remark}
  Note that there are various further developments in connection with
  Alon and Milman's estimate (\ref{eq:AMdiameter}), see, e.g., the work of
  Chung \cite{Chung89}, Mohar \cite{Mohar91}, Chung, Grigor'yan and Yau
  \cite{CGY} and Houdr\'{e} and Tetali \cite{HT01}. In principle, the
  estimate (\ref{eq:AMlambdak}) can be improved accordingly.
\end{remark}

\section{Ratios of higher order Cheeger constants and multi-way
  expanders}\label{section:isoratio}

In this section we derive the following result from the Introduction and
discuss applications in the topic of multi-way expanders.

\begin{corE}
  There exists a universal constant $C$ such that for any graph
  $(G,\mu)$ satisfying $CD(0,\infty)$ and any natural number $k\geq 2$
  we have
  \begin{equation}\label{eq:isoperimetricratioIntroduction}
   h_k(G,\mu)\leq C D_G^{non}D_G^{nor}k\sqrt{\log k}\,h_2(G,\mu).
 \end{equation}
\end{corE}

First, we recall the following results of Lee, Oveis Gharan and
Trevisan \cite[Theorems 1.2, 3.9, and Corollary 4.2]{LGT2013} in our
general setting:

\begin{thm}[Lee, Oveis Gharan and Trevisan]
  Let $(G,\mu)$ be a weighted graph with vertex measure $\mu$. Then we have
  \begin{equation}\label{eq:LGTshifted}
    h_k(G,\mu)\leq C\sqrt{D_G^{non}\log k\lambda_{2k}},
  \end{equation}
  with a universal constant $C > 0$. Moreover, if the graph $G$ has
  genus as most $g \ge 1$ (i.e. $G$ can be embedded into an orientable
  surface of genus at most $g$ without edge crossings), we have
  \begin{equation}\label{eq:LGTgenus}
  h_k(G,\mu)\leq C' \log(g+1)\sqrt{D_G^{non}\lambda_{2k}},
  \end{equation}
  with another universal constant $C' > 0$.
\end{thm}

\begin{proof}[Proof of Corollary \ref{cor:isoperimetricratioIntroduction}]
  Using \eqref{eq:LGTshifted} and Theorems
  \ref{thm:eigenvalueratioIntroduction} and \ref{thm:Buser}, we obtain
\begin{eqnarray*}
  h_k(G,\mu) &\le& C\sqrt{D_G^{non}\log k\lambda_{2k}} \\
  &\le& C' \sqrt{D_G^{non}\log k}\sqrt{D_G^{non}D_G^{nor}}(2k)
\sqrt{\lambda_2(G,\mu)} \\
  &\le& C'' \sqrt{D_G^{non}\log k}\sqrt{D_G^{non}D_G^{nor}}(2k)
\sqrt{D_G^{nor}} h_2(G,\mu) \\
  &=& 2C'' D_G^{non}D_G^{nor} k\sqrt{\log k}\, h_2(G,\mu),
\end{eqnarray*}
with various universal constants $C,C',C''$.
\end{proof}

Moreover, if we replace \eqref{eq:LGTshifted} by \eqref{eq:LGTgenus}
in the above proof, we obtain the following result.

\begin{coro}\label{coro:genus}
  There exists a universal constant $C$ such that if $(G,\mu)$
  satisfies $CD(0,\infty)$, then for any $k\geq 2$,
  \begin{equation}\label{eq:isoperimetricratioPlanar}
    h_k(G,\mu)\leq C D_G^{non}D_G^{nor}\log(g_G+1)kh_2(G,\mu),
  \end{equation}
  where $g_G\geq 1$ is an upper bound of the genus of $G$.
\end{coro}

\begin{remark}
  The order of $k$ in (\ref{eq:isoperimetricratioPlanar}) is
  optimal. This follows from the example of unweighted cycles
  $\mathcal{C}_N$ (which are planar) with the same measure $\mu$ as in
  Example \ref{example:Cycle}, since we have (see e.g. 
 \cite[Proposition 7.3]{Liu13}).
  \begin{equation*}
    h_k(\mathcal{C}_N)=\frac{1}{\left\lfloor\frac{N}{k}\right\rfloor},\,\,\,\text{ for }2\leq k\leq N.
  \end{equation*}
\end{remark}

The dependence on $D_G^{non}D_G^{nor}$ of the ratio estimate is also
necessary. This follows from the following example analyzed in Mimura
\cite{Mimura2014}.

\begin{example}
  Consider the Cartesian product graph $G_{N,2}$ of the unweighted
  complete graphs $\mathcal{K}_N$ and $\mathcal{K}_2$. Assign the
  measure $\mu=\mathbf{1}$ to it. Since complete graphs satisfy
  $CD(0,\infty)$ (in fact the complete graph $\mathcal{K}_N$ is the Cayley graph of ${\mathbb{Z}}/{N}{\mathbb{Z}}$ when all its elements are taken as generators), we know by Theorem
  \ref{thm:CartesianProductIntroduction} that $G_{N,2}$ satisfies
  $CD(0,\infty)$. It is straightforward to see that $h_2(G_{N,2})\leq
  1$. Observe that we can partition $G_{N,2}$ into two induced
  subgraphs $\mathcal{K}_N$ and $\mathcal{K}_N'$. By Lemma 1 of Tanaka
  \cite{Tanaka13} (see also \cite{Mimura2014}), we
  have $$h_3(G_{N,2})\geq h_2(\mathcal{K}_N)=\frac{N}{2}.$$ (Note that
  Tanaka's lemma was stated for the constants $\{\mathfrak{h}_k(G)\}$
  defined below. One can check that it also works for $\{h_k(G)\}$
  here.) Therefore, we obtain
  \begin{equation} \label{eq:exampleMimura}
    \frac{h_3(G_{N,2})}{h_2(G_{N,2})}\geq \frac{N}{2}=\frac{1}{2}d_G.
  \end{equation}
  This shows the necessity of the dependence on the term
  $D_G^{non}D_G^{nor}=d_G$. Note (\ref{eq:exampleMimura}) also holds
  for the normalized measure $\mu$. We comment that one can also
  analyse the eigenvalues of this example to show the necessity of the
  dependence on the degree in (\ref{eq:eigenvalueratio}) (see also
  \cite{Mimura2014}) using an interlacing theorem or Lemma 6 of
  \cite{Tanaka13}.
\end{example}

Now we restrict our considerations to the setting $w=\mathbf{1}_E$ and
$\mu=\mathbf{1}_V$, that is, $G=(V,E)$ is now an unweighted graph with
non-normalized Laplacian. Recently, the concept of multi-way expanders
was defined and studied in Tanaka \cite{Tanaka13} and Mimura
\cite{Mimura2014}.
We denote $\mathfrak{h}_k(G)$ to be the following larger $k$-way
isoperimetric constant (compare with Definition
\ref{defn:multi-wayIso})
\begin{equation}
  \mathfrak{h}_k(G):=\min_{S_1,\ldots, S_k}\max_{1\leq i\leq k}\phi_{1,\mathbf{1}}(S_i),
\end{equation}
where the minimum is taken over all partitions of $V$,
i.e. $V=\bigsqcup_{i=1}^kS_i$, $S_i\neq\emptyset$ for all $i$.

\begin{definition}[Multi-way expanders
  \cite{Tanaka13,Mimura2014}] \label{defn:multi-wayExpanders}
  Let $k \ge 2$ be a natural number. A sequence of finite graphs
  $\{G_m=(V_m, E_m)\}_{m\in \mathbb{N}}$ is called a sequence of
  $k$-way expanders if we have (i) $\sup_md_{G_m}<\infty$; (ii)
  $\lim_{m\rightarrow \infty}|V_m|=\infty$; (iii)
  $\inf_m\mathfrak{h}_k(G_m)>0$.
\end{definition}

Observe that $2$-way expander families coincide with classical
families of expanders. In general, the property of being $(k+1)$-way
expanders is strictly weaker than being $k$-way expanders (see
\cite{Mimura2014}). However,
Mimura \cite{Mimura2014} proved that the concepts of $k$-way expanders
for all $k\geq 2$ are equivalent within the class of finite,
connected, vertex transitive graphs.

As a consequence of Corollary \ref{cor:isoperimetricratioIntroduction}, we have

\begin{coro}\label{thm:multi-way expanders}
  For the class of finite connected graphs satisfying $CD(0,\infty)$,
  the concepts of $k$-way expanders for all $k\geq 2$ are equivalent.
\end{coro}

\begin{proof}
  Using the relation (see Theorem 3.8 in \cite{LGT2013},
  \cite{Mimura2014})
  \begin{equation}
    h_k(G)\leq \mathfrak{h}_k(G)\leq kh_k(G),
  \end{equation}
  and employing Corollary \ref{cor:isoperimetricratioIntroduction} yields
  \begin{equation}
    \mathfrak{h}_k(G)\leq Cd_Gk^2\sqrt{\log k}\mathfrak{h}_2(G).
  \end{equation}
  Hence, when $d_G<\infty$, $\inf_m\mathfrak{h}_k(G_m)>0$ implies
  $\inf_m\mathfrak{h}_2(G_m)>0$. This completes the proof.
\end{proof}

Abelian Cayley graphs lie in the intersection of the class of vertex
transitive graphs and the class of graphs satisfying
$CD(0,\infty)$. It is well known that there are no expanders in the
class of abelian Cayley graphs (see Alon and Roichman
\cite{AR94}). Moreover, Friedman, Murty and Tillich \cite{FMT06} proved an
explicit upper estimate for $\lambda_2$ which implies this
fact. Therefore, we also obtain the following explicit upper estimate
for $\lambda_k$ implying the nonexistence of sequences of multi-way
expanders in this class of abelian Cayley graphs.

\begin{coro}
  For any abelian Cayley graph $G = (V,E)$ of degree $d$ of size
  $N =|V|$, there exists a universal constant $C$ such that for any
  $k\geq 2$,
  \begin{equation}
    \lambda_k(G) \leq Ck^2d^2N^{-\frac{4}{d}}.
  \end{equation}
\end{coro}

This is a direct consequence of Theorem
\ref{thm:eigenvalueratioIntroduction} and the estimate
$\lambda_2\leq CdN^{-\frac{4}{d}}$ in \cite{FMT06}. Therefore, it is
natural to ask the following question.

\begin{quest}
  Does there exist a sequence of expanders satisfying $CD(0,\infty)$?
\end{quest}

We are inclined to a negative answer. For example, the nonexistence of
expander families satisfying $CD(0,\infty)$ would follow if one could
prove that every graph of vertex degree at most $d$ and satisfying
$CD(0,\infty)$ possesses polynomial volume growth with degree
depending only on $d$. In fact, a sequence of expanders, in the contrast,  have exponential volume growth as their Cheeger constant has uniformaly positive lower bound.

\section*{Appendix}\label{section:appendix}

\renewcommand\thesection{\Roman{section}}

\setcounter{section}{1}
\setcounter{thm}{0}

\subsection{Curvature matrix of the triangle and tetrahedron
  graphs}\label{section:appendixTT}

The curvature matrix $\Gamma_2(x)$ for the graph
$(\triangle_{xyz},\mu)$ in Figure \ref{F1} is
\begin{equation*}
  \frac{1}{4C}\left(
  \begin{array}{ccc}
    \frac{3a^2}{B}+\frac{3b^2}{A}+\frac{(a+b)^2}{C} & \frac{bc}{A}-\frac{a(3a+c)}{B}-\frac{a(a+b)}{C} & \frac{ac}{B}-\frac{b(3b+c)}{A}-\frac{b(a+b)}{C}\\
  \frac{bc}{A}-\frac{a(3a+c)}{B}-\frac{a(a+b)}{C}   & \frac{bc}{A}+\frac{3a(a+c)}{B}+\frac{a(a-b)}{C} & \frac{2ab}{C}-\frac{2ac}{B}-\frac{2bc}{A}\\
 \frac{ac}{B}-\frac{b(3b+c)}{A}-\frac{b(a+b)}{C}    & \frac{2ab}{C}-\frac{2ac}{B}-\frac{2bc}{A} & \frac{3b(b+c)}{A}+\frac{b(b-a)}{C}+\frac{ac}{B} \\
  \end{array}
\right).
\end{equation*}

Let us have a closer look at the special case that $A=B=C=1$ and
$a=c$. Then the matrix $4\Gamma_2(x)=4\Gamma_2(z)$ reduces to
\begin{equation*}
\left(
  \begin{array}{ccc}
    4a^2+2ab+4b^2 & -5a^2 & a^2-2ab-4b^2 \\
    -5a^2 & 7a^2 & -2a^2 \\
     a^2-2ab-4b^2 & -2a^2 & a^2+2ab+4b^2 \\
  \end{array}
\right),
\end{equation*}
and the matrix $4\Gamma_2(y)$ is
\begin{equation*}
 \left(
   \begin{array}{ccc}
     10a^2 & -5a^2 & -5a^2 \\
     -5a^2 & 3a^2+4ab & 2a^2-4ab \\
     -5a^2 & 2a^2-4ab & 3a^2+4ab \\
   \end{array}
 \right).
\end{equation*}
Observe that when $b\geq a/2$, the above two matrices are both
diagonally dominant and hence positive-semidefinite. In fact they are
always positive-semidefinite for any $a,b\geq 0$.

The matrix $4A^2\Gamma_2(x)$ for the tetrahedron graph $(T_4,\mu)$ in
Figure \ref{F2} is given by
\begin{equation*}
  \begin{pmatrix}
    \begin{smallmatrix}
      2ab+2ac+2bc+4a^2+4b^2+4c^2 & -2ab+2ac-2bc-4b^2 & 2ab-2ac-2bc-4c^2 & -2ab-2ac+2bc-4a^2 \\
      -2ab+2ac-2bc-4b^2& 2ac+2ab+2bc+4b^2 & -2ab-2ac+2bc & -2bc-2ac+2ab \\
      2ab-2ac-2bc-4c^2    & -2ab-2ac+2bc &  2ab+2ac+2bc+4c^2& -2ab+2ac-2bc \\
      -2ab-2ac+2bc-4a^2    & -2bc-2ac+2ab &  -2ab+2ac-2bc& 2ab+2ac+2bc+4a^2 \\
    \end{smallmatrix}
  \end{pmatrix}.
\end{equation*}
This is a positive-semidefintie matrix.

\subsection{Proof of Buser inequality}\label{section:appendixBuser}

Buser's inequality in the graph theoretical setting for the
non-normalized unweighted Laplacian, motivated by the original proof
by Ledoux \cite{Ledoux04} for manifolds, was given in Klartag, Kozma,
Ralli and Tetali \cite{KlarKozma}. Since we need in this article a
weighted version, we present a self-contained proof in this general
setting, for the readers convenience.

Ledoux' approach is based on heat semigroup techniques,
so we start with basic facts on the continuous time heat equation,
\begin{equation}\label{eq:heateq}
  \left\{
    \begin{array}{ll}
      \frac{\partial}{\partial t}u(x,t)=\Delta u(x,t),\\
      u(x,0)=f(x).
    \end{array}
  \right.
\end{equation}
This is in fact a matrix differential equation.  Its solution $u:
V\times [0,\infty)\rightarrow \mathbb{R}$ can be written as
$u(x,t)=P_tf(x)$ where $P_t:=e^{t\Delta}$.  Let us choose an
orthonormal basis $\{\psi_i\}_{i=1}^{N}$ of the space $l^2(V)$ (the
function space defined by the inner product $(f,g)_\mu:=\sum_{x\in
  V}f(x)g(x)\mu(x)$), consisting of eigenfunctions of $\Delta$. One
can derive the following properties from the definition.

\begin{pro}[see, e.g., \cite{BHLLMY13,Chung}]\label{pro:heatsemigroup}
  The operator $P_t, t\geq 0$ satisfies the following properties:
  \begin{enumerate}
  \item $P_t$ is a self-adjoint operator;
  \item $P_t$ commutes with $\Delta$, i.e. $P_t\Delta=\Delta P_t$;
  \item $P_tP_s=P_{t+s}$ for any $t,s\geq 0$;
  \item $P_tf(x)=\sum_{y\in V}f(y)p_t(x,y)\mu(y)$,
    where $$p_t(x,y)=\sum_{i=1}^Ne^{-\lambda_it}\psi_i(x)\psi_i(y)\geq
    0$$ and $\sum_{y\in V}p_t(x,y)\mu(y)=1$. In particular, $0\leq
    P_t(\chi_S)\leq 1$, where $\chi_S$ is the characteristic function
    of a subset $S\subset V$;
  \item $\sum_{x\in V}P_tf(x)\mu(x)=\sum_{x\in V}f(x)\mu(x)$.
  \end{enumerate}
\end{pro}

The following Bakry-\'{E}mery type gradient estimate of $P_tf$ is an
important feature of the CD-inequality.

\begin{lemma}[see e.g. \cite{Bakry}, \cite{Ledoux04}]\label{lemma:BEgradient}
  $(G, \mu)$ satisfies $CD(-K,\infty)$ if and only if, for any
  function $f:V\rightarrow \mathbb{R}$, the following holds,
  \begin{equation}\label{eq:BEgradient}
    \Gamma(P_tf)\leq e^{2Kt}P_t(\Gamma(f)).
  \end{equation}
\end{lemma}

\begin{proof}
  The proof in \cite{Bakry} (see Proposition 3.3 there) or \cite{Ledoux04} (see (5.3) there) works still
  for the graph setting. For any $0\leq s\leq t$, define
  \begin{equation}
    F(s):=e^{2Ks}P_s(\Gamma(P_{t-s}f)).
  \end{equation}
  Observe that $F(0)=\Gamma(P_tf)$ and
  $F(t)=e^{2Kt}P_t(\Gamma(f))$.
  We calculate
  \begin{equation*}
    \frac{d}{ds}F(s)=2Ke^{2Ks}P_s(\Gamma(P_{t-s}f))+e^{2Ks}\Delta P_s(\Gamma(P_{t-s}f))+e^{2Ks}P_s(\frac{d}{ds}\Gamma(P_{t-s}f)).
  \end{equation*}
  Recalling the equations (\ref{eq:gradient}) and (\ref{eq:heateq}),
  it is straightforward to see for any $x\in V$
  \begin{align*}
    \frac{d}{ds}\Gamma(P_{t-s}f)(x)&=\frac{1}{\mu(x)}\sum_{y,y\sim
      x}w_{xy}(P_{t-s}f(y)-P_{t-s}f(x))(-\Delta P_{t-s}f(y)+\Delta
    P_{t-s}f(x))\\&=-2\Gamma(P_{t-s}f, \Delta P_{t-s}f)(x).
  \end{align*}
  Now we observe that if $(G,\mu)$ satisfies $CD(-K, \infty)$,
  \begin{equation*}
    \frac{d}{ds}F(s)=2e^{2Ks}P_s(\Gamma_2(P_{t-s}f)+K\Gamma(P_{t-s}f))\geq 0,
  \end{equation*}
  where we used (\ref{eq:Gammatwo}). This
  implies (\ref{eq:BEgradient}).

  On the other hand, if (\ref{eq:BEgradient}) holds, by considering the Taylor expansions at $t=0$, we have
  \begin{equation}
  \Gamma(f+t\Delta f+o(t))\leq (1+2Kt+o(t))(\Gamma(f)+t\Delta\Gamma(f)+o(t)).
  \end{equation}
  After aggregating, we obtain
   \begin{equation}
  t(\Delta\Gamma(f)-2\Gamma(f,\Delta f)+o(t))+2Kt\Gamma(f)\geq 0.
  \end{equation}
  Dividing by $2t$ and letting $t$ tends to zero, we derive $$\Gamma_2(f)\geq -K\Gamma(f).$$
  Since this holds for any function $f$, we prove that $(G,\mu)$ satisfies $CD(-K,\infty)$.
\end{proof}

The following Lemma can be considered as a reverse Poincar\'{e}
inequality.

\begin{lemma}\label{lemma:revPoin}
  Assume that $(G,\mu)$ satisfies $CD(0,\infty)$. Then, we have for
  any function $f:V\rightarrow \mathbb{R}$, any $t\geq 0$, and any
  $x\in V$,
  \begin{equation}
    P_t(f^2)(x)-(P_tf)^2(x)\geq 2t\Gamma(P_tf)(x).
  \end{equation}
\end{lemma}

\begin{proof}
  For $0\leq s\leq t$, set $G(s)=P_s((P_{t-s}f)^2)$.  Then we have
  $G(0)=(P_tf)^2$, $G(t)=P_t(f^2)$.  Using the gradient estimate in
  Lemma \ref{lemma:BEgradient}, we have
  \begin{align*}
    \frac{d}{ds}G(s)&=\Delta P_s((P_{t-s}f)^2)+P_s(-2P_{t-s}f\Delta P_{t-s}f)\\
    &=2P_s(\Gamma(P_{t-s}f))\geq 2\Gamma(P_tf).
  \end{align*}
  Now we arrive at
  \begin{equation*}
    G(t)-G(0)=\int_0^t \frac{d}{ds}G(s) ds\geq 2\Gamma(P_tf)\int_0^tds= 2t\Gamma(P_tf).
  \end{equation*}
  This completes the proof.
\end{proof}

Define the $l_p$ norm of a function $f:V\rightarrow \mathbb{R}$ as
$\Vert f\Vert_p:=\left(\sum_{x\in
    V}\mu(x)|f(x)|^p\right)^{\frac{1}{p}}$.  We have the following
direct corollary.

\begin{coro}\label{coro:l1norm}
  Assume that $(G,\mu)$ satisfies $CD(0,\infty)$. Then, we have for
  any function $f:V\rightarrow \mathbb{R}$ and any $t\geq 0$,
  \begin{equation*}
    \Vert f-P_tf\Vert_1\leq \sqrt{2t}\Vert\sqrt{\Gamma(f)}\Vert_1.
  \end{equation*}
\end{coro}

\begin{proof}
  Note first that, by Lemma \ref{lemma:revPoin} and Proposition
  \ref{pro:heatsemigroup} (4), we have
  \begin{equation}\label{eq:inftynorm}
    \Vert\sqrt{\Gamma(P_tf)}\Vert_{\infty}\leq \frac{1}{\sqrt{2t}}\Vert f\Vert_{\infty}.
  \end{equation}
  Let $g:V\rightarrow \mathbb{R}$ be $g(x):={\rm
    sgn}(f(x)-P_tf(x))$. Then we obtain
  \begin{align*}
    \Vert f-P_tf\Vert_1&=\sum_{x\in V}\mu(x)(f(x)-P_tf(x))g(x)=-\sum_{x\in V}\mu(x)\int_0^t\Delta P_sf(x)ds g(x)\\
    &=-\int_0^t\sum_{x\in V}\mu(x)P_sg(x)\Delta f(x)ds=\int_0^t\sum_{x\in V}\mu(x)\Gamma(P_sg(x),f(x))ds\\
    &\leq \int_0^t\sum_{x\in V}\mu(x)\sqrt{\Gamma(P_sg)(x)}\sqrt{\Gamma(f)(x)}ds\\&\leq \Vert\sqrt{\Gamma(f)}\Vert_1\int_0^t\Vert\sqrt{\Gamma(P_sg)}\Vert_{\infty}ds\leq \Vert\sqrt{\Gamma(f)}\Vert_1\Vert g\Vert_{\infty}\int_0^t\frac{1}{\sqrt{2s}}ds\\
    &\leq \sqrt{2t}\Vert\sqrt{\Gamma(f)}\Vert_1,
  \end{align*}
  where we used Proposition \ref{pro:heatsemigroup} (1-2),
  (\ref{eq:summaBypart}), (\ref{eq:gammaHolder}) and
  (\ref{eq:inftynorm}).
\end{proof}

Recall that $\chi_S$ denotes the characteristic function of $S\subset
V$.

\begin{lemma}\label{lemma:boundaryMeasure}
  We have
  \begin{equation}
    \Vert\sqrt{\Gamma(\chi_S)}\Vert_1\leq\sqrt{2D_G^{nor}}E(S,V\setminus S).
  \end{equation}
\end{lemma}

\begin{proof}
  This follows from the direct calculation given here:
  \begin{align*}
    \Vert\sqrt{\Gamma(\chi_S)}\Vert_1&=\sum_{x\in V}\mu(x)\sqrt{\frac{1}{2\mu(x)}\sum_{y,y\sim x}w_{xy}(\chi_S(x)-\chi_S(y))^2}\\
    &\leq \sum_{x\in V}\sqrt{\frac{\mu(x)}{2}}\sum_{y,y\sim x}\sqrt{w_{xy}}|\chi_S(x)-\chi_S(y)|\\
    &\leq \sqrt{\frac{D_G^{nor}}{2}}\sum_{x\in V}\sum_{y,y\sim x}w_{xy}|\chi_S(x)-\chi_S(y)|\\
    &=\sqrt{2D_G^{nor}}E(S,V\setminus S).
  \end{align*}
\end{proof}

\begin{proof}[Proof of (\ref{eq:BuserCD0})]
  Using Corollary \ref{coro:l1norm} and Lemma
  \ref{lemma:boundaryMeasure}, we have
$$\sqrt{2t}\Vert\sqrt{\Gamma(\chi_S)}\Vert_1\leq 2\sqrt{D_G^{nor}t}E(S, V\setminus S)$$
and
  \begin{align}
    &\sqrt{2t}\Vert\sqrt{\Gamma(\chi_S)}\Vert_1\notag\\
    \geq&\sum_{x\in V}\mu(x)|\chi_S(x)-P_t\chi_S(x)|=\sum_{x\in S}\mu(x)(1-P_t\chi_S(x))+\sum_{x\in V\setminus S}\mu(x)P_t\chi_S(x)\notag\\
    =&2\mu(S)-2\sum_{x\in S}P_t(\chi_S)(x)\mu(x),\label{eq:tobeInset}
  \end{align}
  where we used Proposition \ref{pro:heatsemigroup} (4-5).

  Now let $\{\alpha_i\}_{i=1}^N$ be $N$ constants such that
  $\chi_S=\sum_{i=1}^N\alpha_i\psi_i$, where $\{\psi_i\}_{i=1}^{N}$
  are the orthonormal basis of $l^2(V,\mu)$ given by eigenfunctions with the choice $\psi_1\equiv 1/\sqrt{\mu(V)}$.
  Then we have $\Vert\chi_S\Vert_2^2=\sum_{i=1}^N\alpha_i^2=\mu(S)$
  and
  \begin{equation*}
    \alpha_1=(\chi_S,\psi_1)=\sum_{x\in V}\mu(x)\chi_S(x)\frac{1}{\sqrt{\mu(V)}}=\frac{\mu(S)}{\sqrt{\mu(V)}}.
  \end{equation*}
  Now we have by Proposition \ref{pro:heatsemigroup} (4)
  \begin{align*}
    &\sum_{x\in S}P_t(\chi_S)(x)\mu(x)\\
=&\sum_{x\in V}\chi_S(x)P_t(\chi_S)(x)\mu(x)=\sum_{i=1}^Ne^{-\lambda_it}\alpha_i^2\leq e^{-\lambda_2t}\sum_{i=2}^N\alpha_i^2+\alpha_0^2\\
    =&e^{-\lambda_2t}\left(\mu(S)-\frac{\mu(S)^2}{\mu(V)}\right)+\frac{\mu(S)^2}{\mu(V)}.
  \end{align*}
  Inserting the above estimate into (\ref{eq:tobeInset}), we arrive at
  \begin{equation}
    2\sqrt{D_G^{nor}t}E(S, V\setminus S)\geq 2\left(\mu(S)-\frac{\mu(S)^2}{\mu(V)}\right)(1-e^{-\lambda_2t}).
  \end{equation}
  Taking $t=\frac{1}{\lambda_2}$, we obtain for those $S$ with
  $\mu(S)\leq \frac{1}{2}\mu(V)$,
  \begin{equation*}
    2\sqrt{\frac{D_G^{nor}}{\lambda_2}}E(S, V\setminus S)\geq \mu(S)(1-e^{-1}).
  \end{equation*}
  This completes the proof.
\end{proof}

\subsection{CD-inequalities of dumbbell
  graphs}\label{section:appendixDumbell}

In this subsection we present the calculations for the CD-inequalities
of dumbbell graphs $G_N$ claimed in Example
\ref{example:dumbbell}. They are modified from that of
\cite[Proposition 3]{JostLiu14}.

A general formula representing $\Gamma_2(f)$ is given by
\begin{align}
  \Gamma_2(f)(x)=& Hf(x)+\frac{1}{2}(\Delta f(x))^2-\frac{1}{2}\frac{\sum_{y,y\sim x}w_{xy}}{\mu(x)}\Gamma(f)(x)\notag\\
  &-\frac{1}{4}\frac{1}{\mu(x)}\sum_{y,y\sim
    x}w_{xy}(f(y)-f(x))^2\frac{\sum_{z,z\sim
      y}w_{yz}}{\mu(y)},\label{eq:gamma2general}
\end{align}
where
\begin{equation*}
Hf(x):=\frac{1}{4}\frac{1}{\mu(x)}\sum_{y,y\sim x}\frac{w_{xy}}{\mu(y)}\sum_{z,z\sim y}w_{yz}(f(x)-2f(y)+f(z))^2.
\end{equation*}
This is an extension of \cite[(2.9)]{JostLiu14} to our general
setting $(G,\mu)$.

Let us first consider the case of the unweighted normalized
Laplacian. Let $x$ be a vertex of $G_N$ which is different from $y_0$
or $y'_0$ (see Figure \ref{F3}). First observe that
\begin{equation*}
  Hf(x)\geq \frac{1}{4N(N-1)}\sum_{y,y\sim x}\sum_{\substack{z,z\sim y\\z\neq y'_0}}(f(x)-2f(y)+f(z))^2.
\end{equation*}
Now our calculations reduce to the complete graph $\mathcal{K}_N$
itself. Note that when $y,z\neq x$,
\begin{align*}
  &(f(x)-2f(y)+f(z))^2+(f(x)-2f(z)+f(y))^2\\
  =&(f(x)-f(y))^2+(f(x)-f(z))^2+4(f(y)-f(z))^2.
\end{align*}
Then we have
\begin{align*}
  Hf(x)\geq
  \frac{N+2}{2N}\Gamma(f)(x)+\frac{1}{N(N-1)}\sum_{(y,z)}(f(y)-f(z))^2,
\end{align*}
where the second summation is over all unordered pair of neighbors of
$x$. By (\ref{eq:gamma2general}), we arrive at
\begin{equation*}
  \Gamma_2(f)(x)\geq \frac{2-N}{2N}\Gamma(f)(x)+\frac{1}{2}\left(\Delta f(x)\right)^2+\frac{1}{N(N-1)}\sum_{(y,z)}(f(y)-f(z))^2.
\end{equation*}
The last two terms above can be further manipulated as follows,
\begin{align*}
  &\frac{1}{2(N-1)^2}\left(\sum_{y,y\sim x}(f(y)-f(x))\right)^2+\frac{1}{N(N-1)}\sum_{(y,z)}(f(y)-f(z))^2\\
  \geq &\frac{1}{N(N-1)}\Bigg[\frac{1}{2}\sum_{y,y\sim x}(f(y)-f(x))^2-\sum_{(y,z)}(f(y)-f(x))(f(z)-f(x))\\
  &\hspace{2cm}+\sum_{(y,z)}\left((f(y)-f(x))^2+(f(z)-f(x))^2\right)\Bigg]\\
  =&\frac{1}{N(N-1)}\Bigg[\left(\frac{1}{2}+\frac{N-2}{2}\right)\sum_{y,y\sim x}(f(y)-f(x))^2+\frac{1}{2}\sum_{(y,z)}(f(y)-f(z))^2\Bigg]\\
  \geq&\frac{N-1}{N}\Gamma(f)(x).
\end{align*}
In the equality above, we use the facts that
\begin{align*}&\frac{1}{2}\sum_{(y,z)}\left((f(y)-f(x))^2+(f(z)-f(x))^2\right)-\sum_{(y,z)}(f(y)-f(x))(f(z)-f(x))\\
=&\frac{1}{2}\sum_{(y,z)}(f(y)-f(z))^2
\end{align*}
and
\begin{align*}&\frac{1}{2}\sum_{(y,z)}\left((f(y)-f(x))^2+(f(z)-f(x))^2\right)\\
=&\frac{1}{4}\sum_{y,y\sim x}\sum_{z,z\sim x, z\neq y}(f(y)-f(x))^2+\frac{1}{4}\sum_{z,z\sim x}\sum_{y,y\sim x, y\neq z}(f(z)-f(x))^2\\
=&\frac{N-2}{2}\sum_{y,y\sim x}(f(y)-f(x))^2.
\end{align*}
Therefore we have $\Gamma_2(f)(x)\geq\frac{1}{2}\Gamma(f)(x)$. That
is, $G_N$ satisfies $CD\left(\frac{1}{2},\infty\right)$ at any vertex
$x\neq y_0, y_0'$.

\begin{remark}
  We note that this CD-inequality at vertex $x$ still holds even if we
  attach different graphs to every vertex in $\mathcal{K}_N$ other
  than $x$ via single edges.
\end{remark}

At $y_0$, $CD(0,\infty)$ does not hold. Let $f_0$ be the function
taking the value $1$ at $y_0'$, $2$ at all other vertices in
$\mathcal{K}_N'$, and $0$ at all vertices in $\mathcal{K}_N$. Then one
can check by (\ref{eq:gamma2general})
that
$$\Gamma_2(f_0)(y_0)=\frac{3-N}{2N^2}<0, \,\,\,\text{if }N \geq 4.$$
In the case $N=3$, we can use another function $g_0$ taking the value
$1$ at $y_0$, $-1$ at all other vertices in $\mathcal{K}_3$, $4$ at
$y_0'$, and $7$ at other vertices in $\mathcal{K}_3'$. One can then
check directly that $\Gamma_2(g_0)(y_0)=-\frac{1}{9}<0$.

For the case of the unweighted non-normalized Laplacian, the
calculations are similar. Note in this case at $x\neq y_0,y_0'$, we
have
\begin{align*}
  \Gamma_2(f)(x)=&Hf(x)+\frac{1}{2}\left(\Delta f(x)\right)^2-\frac{d_x}{2}\Gamma(f)(x)-\frac{1}{4}\sum_{y,y\sim x}(f(y)-f(x))^2d_y\\
  \geq &Hf(x)+\frac{1}{2}\left(\Delta f(x)\right)^2-N\Gamma(f)(x).
\end{align*}
Carrying out the calculation in the same way as in the normalized case
we finally conclude $\Gamma_2(f)(x)\geq \frac{N}{2}\Gamma(f)(x)$. The
arguments for CD-inequalities at $y_0, y_0'$ can be done with the same
special functions as in the normalized case.

\section*{Acknowledgements}
We thank Frank Bauer for valuable comments about Buser's inequality on
graphs in \cite{BHLLMY13} and \cite{KlarKozma}. We are also grateful
to the anonymous referees for many useful suggestions for improvement
of this article. This work was supported by the EPSRC Grant
EP/K016687/1 "Topology, Geometry and Laplacians of Simplicial
Complexes".

\end{document}